\documentclass[10pt]{amsart}
\theoremstyle{plain}
\usepackage{amsmath, amsfonts, amsthm, amssymb,amscd,bbold}
\usepackage{graphics}
\usepackage{color}
\usepackage{epsfig}
\usepackage[all]{xy}
\newcommand{\ba}{\mathbf{a}}
\newcommand{\bb}{\mathbf{b}}

\graphicspath{ {Figures/} }

\title[Gradings]{On Gradings in Khovanov homology and sutured Floer homology}

\author{J. Elisenda Grigsby}
\thanks{JEG was partially supported by a Viterbi-endowed MSRI postdoctoral fellowship and NSF grant number DMS-0905848.}
\address{Boston College; Department of Mathematics; 301 Carney Hall; Chestnut Hill, MA 02467}
\email{grigsbyj@bc.edu}

\author{Stephan M. Wehrli}
\thanks{SMW was partially supported by an MSRI postdoctoral fellowship.}
\address{Syracuse University; Mathematics Department; 215 Carnegie; Syracuse, NY 13244}
\email{smwehrli@syr.edu}

\theoremstyle{plain}
\newtheorem{theorem}{Theorem}[section]
\newtheorem{lemma}[theorem]{Lemma}

\newtheorem{construction}[theorem]{Construction}

\theoremstyle{definition}

\newcommand{\Szabo}{{Szab{\'o}} }
\newcommand{\Juhasz}{{Juh{\'a}sz} }

\newcommand{\R}{\ensuremath{\mathbb{R}}}
\newcommand{\Z}{\ensuremath{\mathbb{Z}}}

\newcommand{\bL}{\ensuremath{\mathbb{L}}}

\newcommand{\boldalpha}{\ensuremath{\mbox{\boldmath $\alpha$}}}
\newcommand{\boldbeta}{\ensuremath{\mbox{\boldmath $\beta$}}}

\newcommand{\boldSigma}{\ensuremath{\mbox{\boldmath $\Sigma$}}}
\newcommand{\cP}{\ensuremath{\mathcal{P}}}
\newcommand{\cI}{\ensuremath{\mathcal{I}}}

\newcommand{\spincs}{\ensuremath{\mathfrak{s}}}

\begin{document}
\bibliographystyle{plain}

\begin{abstract} We discuss generalizations of Ozsv{\'a}th-Szab{\'o}'s spectral sequence relating Khovanov homology and Heegaard Floer homology, focusing attention on an explicit relationship between natural $\Z$ (resp., $\frac{1}{2}\Z$) gradings appearing in the two theories.  These two gradings have simple representation-theoretic (resp., geometric) interpretations, which we also review.
\end{abstract}
\maketitle

\section{Introduction}
The Ozsv{\'a}th-\Szabo spectral sequence \cite{MR2141852} relating Khovanov homology \cite{MR1740682} and Heegaard-Floer homology \cite{MR2113019} has inspired an abundance of work in these and related fields (see, e.g., \cite{MR2250492}, \cite{MR2275002}, \cite{GT07060741}, \cite{GT07104300}, \cite{MR2601010}, \cite{MR2576686}, \cite{GT08054423}, \cite{GT08071341}, \cite{GT08082817}, \cite{GT08093293}, \cite{MR2652216}, \cite{GT09033746}, \cite{GT09090816}, \cite{MR2608442}, \cite{SurfDecomp}, \cite{AnnularLinks}), yielding applications in areas as disparate as contact geometry \cite{MR2652216} and Dehn surgery \cite{GT08071341}.  A particularly stunning result with roots in the circle of ideas found in \cite{MR2141852} is Kronheimer-Mrowka's proof that Khovanov homology detects the unknot \cite{GT10054346}.

The Ozsv{\'a}th-\Szabo spectral sequence has been extended, by L. Roberts in  \cite{GT07060741} and the authors in \cite{MR2601010, AnnularLinks}, to yield a relationship between the ``sutured" versions of Khovanov homology \cite{MR2113902, MR2124557} and Heegaard Floer homology \cite{MR2253454} one naturally associates to tangles in certain simple sutured manifolds.  Applying Gabai's sutured manifold technology \cite{MR723813} to these homology theories, an idea pioneered by \Juhasz \cite{MR2390347}, reveals that the Ozsv{\'a}th-\Szabo spectral sequence exhibits a great deal more algebraic structure than originally believed.  See \cite{SurfDecomp} as well as \cite{MR2601010, AnnularLinks} for more details.

Our goal in this short note is to establish an explicit relationship between certain natural gradings appearing in Khovanov homology and sutured Heegaard-Floer homology.  This relationship is already hinted at in the work of Plamenevskaya \cite{MR2250492} and implicit in the work of L. Roberts \cite{GT07060741}.  Our motivation here is simply to provide an easily--digestible account, complete with the appropriate representation-theoretic and geometric context, for the benefit of the Khovanov and Heegaard-Floer communities, for whom the result appears to be of independent interest.
{\flushleft {\bf Acknowledgements:}} We thank Denis Auroux and Catharina Stroppel for a number of interesting discussions, as well as the MSRI spring 2010 semester-long program on Knot Homology Theories for making these discussions possible.  The relationship we describe here between Floer-theoretic and representation-theoretic gradings was observed independently (in a somewhat different context) by Auroux and Stroppel.  The first author would also like to thank the organizers of Jaco's 70th Birthday Conference for an enjoyable and enlightening weekend.  Happy birthday, Bus!

\section{Background on Ozsv{\'a}th-\Szabo spectral sequence} \label{sec:background}
Let $\bL \subset S^3$ be a link, $\overline{\bL} \subset S^3$ its mirror, and $Y$ a closed, connected, oriented $3$--manifold.  In what follows, $\widetilde{Kh}(\bL)$ will denote the reduced Khovanov homology of $\bL$, defined in \cite{MR2034399}, and $\widehat{HF}(Y)$ will denote the (hat version of) the Heegaard-Floer homology of $Y$, defined in \cite{MR2113019}.  Furthermore, if $(B,\partial B) \subset (X,\partial X)$ is a properly-imbedded codimension $2$ submanifold of a manifold $X$, then $\boldSigma(X,B)$ will denote the double-branched cover of $X$, branched over $B$, and $\widetilde{B}$ will denote its preimage, $p^{-1}(B)$, under the covering map $p: \boldSigma(X,B) \rightarrow X$.  Throughout the paper, all Khovanov and Heegaard-Floer homology theories will be considered with $\mathbb{F} = \mathbb{Z}/2\mathbb{Z}$ coefficients.  Ozsv{\'a}th-\Szabo proved:

\begin{theorem} \cite[Thm. 1.1]{MR2141852} \label{thm:OzSz} Let $\bL \subset S^3$ be a link.  There is a spectral sequence whose $E^2$ term is $\widetilde{Kh}(\overline{\bL})$ and whose $E^\infty$ term is $\widehat{HF}(\boldSigma(S^3,\bL))$.
\end{theorem}

\begin{proof}[Sketch of Proof] Ozsv{\'a}th-\Szabo built this spectral sequence by constructing a Floer-theoretic ``cube of resolutions" complex for $\widehat{HF}(\boldSigma(S^3,\bL))$ that looks very much like the  chain complex used to define the reduced Khovanov homology of $\overline{\bL}$.

More precisely, recall that if we are given an $n$--crossing diagram for a link $\overline{\bL} \subset S^3$ with crossings labeled $C_1, \ldots, C_n$ along with a marked point $z \subset \left(\bL - \bigcup_{i=1}^n N(C_i)\right)$, we can construct a chain complex computing $\widetilde{Kh}(\bL)$ by the following procedure (see \cite{MR1740682, MR1917056, MR2034399} for more details).
\begin{enumerate}
	\item Associate to each of the $2^n$ vertices of an $n$--dimensional cube a so-called ``complete resolution" of the diagram.  Specifically, the vertices can be put in one-to-one correspondence with the elements of $\{0,1\}^n$; to vertex $\cI \in \{0,1\}^n$ we then associate the resolved diagram whose crossings have been replaced by ``0" and ``1" resolutions (see Figure \ref{fig:res}) according to the prescription $\cI$.
	\item Apply Khovanov's (1+1)--dimensional TQFT to this cube.  This replaces the resolved diagram at vertex $\cI$ with the graded vector space $V^{\otimes k}$, where $V = \mathbb{F}_{-1} \oplus \mathbb{F}_{1}$ (the subscripts correspond to the grading) and $k$ is the number of unmarked components in the resolved diagram at $\cI$.  To the edges of the cube, Khovanov associates linear ``multiplication" and ``comultiplication" maps.
	\item This data defines a bigraded chain complex whose underlying vector space is the direct sum of all of the vector spaces at the vertices of the cube and whose differential is the sum of the linear maps along the edges of the cube.  The homology of this chain complex is a link invariant.
\end{enumerate}

\begin{figure}
\begin{center}
\resizebox{2in}{!}{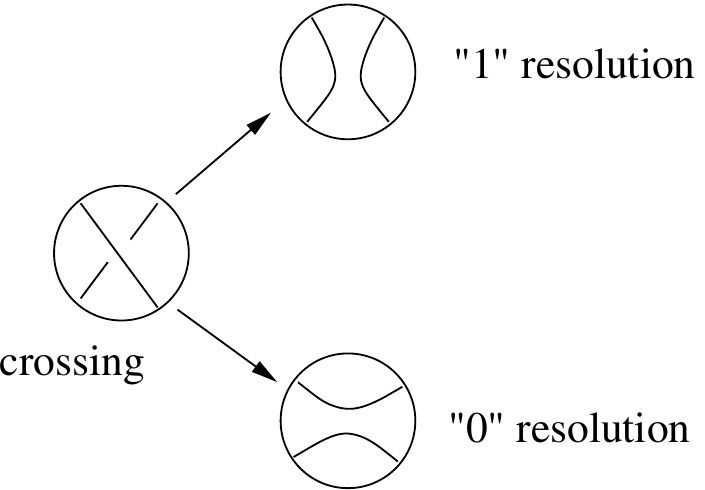}
\end{center}
\caption{Replacing a crossing with a ``$0$" or ``$1$" resolution}
\label{fig:res}
\end{figure}

On the Heegaard-Floer side, Ozsv{\'a}th-\Szabo had previously proved the existence of an exact sequence associated to a framed knot $K$ in a closed, connected, oriented $3$--manifold $Y$ \cite[Thm. 9.12]{MR2113020}.  Namely, there is a long exact sequence relating the triple \[\{\widehat{HF}(K_{\mu}), \widehat{HF}(K_{\lambda}), \widehat{HF}(K_{\lambda + \mu})\},\] where $\mu$ denotes the meridian of $K$ in $Y$, $\lambda$ denotes the framing on $K$, and $K_{\alpha}$ denotes the $3$--manifold obtained by doing $\alpha$--surgery on $K \subset Y$ for $\alpha = \mu, \lambda, \lambda+\mu$.  Ozsv{\'a}th-Szab{\'o}'s crucial observation is that the branched double cover of the $3$--ball neighborhood $N(C_i)$ of each crossing $C_i$, branched along $L \cap N(C_i)$, is a solid torus, i.e., the neighborhood of a framed knot $K_i \subset \boldSigma(S^3,\bL)$.  Furthermore, replacing the crossing by the two different resolutions corresponds, in $\boldSigma(S^3,\bL)$, to performing alternate Dehn surgeries on $K_i$, yielding the two other $3$--manifolds related to $\boldSigma(S^3,\bL)$ in the long exact sequence for $K_i$.

Thus one can can compute $\widehat{HF}(\boldSigma(S^3,\bL))$ using a filtered chain complex constructed as an iterated mapping cone associated to the link $K_1 \cup \ldots \cup K_n \subset \boldSigma(S^3,\bL)$ obtained by taking the preimage of the (neighborhoods of the) crossings $C_1, \ldots, C_n$.  More precisely, the existence of the long exact sequence for each $K_i$ allows one to replace (up to homotopy) any chain complex $\widehat{CF}(\boldSigma(S^3,\bL))$ computing $\widehat{HF}(\boldSigma(S^3,\bL))$ with the mapping cone of a chain map \[f_1: \widehat{CF}(\boldSigma(S^3,\bL_{0, \infty, \ldots, \infty})) \rightarrow \widehat{CF}(\boldSigma(S^3,\bL_{1, \infty, \ldots, \infty}))\] where $\bL_{0, \infty, \ldots, \infty}$ (resp., $\bL_{1, \infty, \ldots, \infty}$) represents the link obtained from $\bL$ by replacing the crossing $C_1$ by its $0$--resolution (resp., its $1$--resolution).  Recall that the mapping cone of a chain map $f: \mathcal{C} \rightarrow \mathcal{C}'$ between chain complexes $\mathcal{C}$ and $\mathcal{C}'$ is the chain complex with underlying vector space $\mathcal{C} \oplus \mathcal{C}'\{-1\}$ and differential, $D$, given by the sum of the internal differentials $\partial_{\mathcal{C}}, \partial_{\mathcal{C}'}$ and the chain map $f$: \[D := \left(\begin{array}{cc}
				\partial_{\mathcal{C}} & 0\\
				f & \partial_{\mathcal{C}'}\end{array}\right).\]

Iterating this process over all crossings, one obtains a cube of resolutions similar to the Khovanov cube described above (in particular, the vertices and edge maps can be seen to agree).  However, there are additional higher-order differentials mapping along higher--dimensional faces of the cube.  These arise, at least in part, because each of the three Heegaard-Floer chain complexes in the long exact sequence for a knot is only {\em chain homotopic} (not {\em chain isomorphic}) to the mapping cone of the chain map between the other two.

The filtered chain complex computing $\widehat{HF}(\boldSigma(S^3,\bL))$, described above,  can be constructed explicitly using a so-called {\em Heegaard multi-diagram}.  To understand this, recall that, given a closed, connected, oriented $3$--manifold $Y$, one defines a chain complex computing $\widehat{HF}(Y)$ using the data of a generic self-indexing Morse decomposition of $Y$ with a single critical point of index $0$ and a single critical point of index $3$.  This yields a decomposition
\[Y = \mathcal{H}_{\alpha} \bigcup_{\Sigma_g} -\mathcal{H}_{\beta}\] of $Y$ into two solid handlebodies $\mathcal{H}_{\alpha}$ and $\mathcal{H}_{\beta}$.  One can then choose a corresponding pointed Heegaard diagram $(\Sigma_g, \boldalpha, \boldbeta, z)$, where
\begin{itemize}
	\item $\Sigma_g$ is a genus $g$ Heegaard surface,
	\item $\boldalpha = \{\alpha_1, \ldots, \alpha_g\}$ (resp., $\boldbeta = \{\beta_1, \ldots, \beta_g\}$) is a choice of $g$ circles in $\Sigma_g$, linearly-independent in $H_1(\Sigma_g;\Z)$, bounding a full set of compressing disks in $\mathcal{H}_\alpha$ (resp., $\mathcal{H}_\beta$), and
	\item $z \in \Sigma_g$ is a marked point in the complement of $\boldalpha \cup \boldbeta$.
\end{itemize}
After picking a generic complex structure on $\Sigma_g$ (see \cite[Sec. 3]{MR2113019}) and ensuring that the Heegaard diagram satisfies some additional technical assumptions (see \cite[Sec. 4]{MR2113019}), one then defines a chain complex from the data $(\Sigma_g, \boldalpha, \boldbeta, z)$, whose

\begin{itemize}
	\item generators are given by intersection points \[\mathbb{T}_\alpha \cap \mathbb{T}_\beta = (\alpha_1 \times \ldots \times \alpha_g) \cap (\beta_1 \times \ldots \times \beta_g) \subset Sym^g(\Sigma_g),\] where $Sym^g(\Sigma_g) := \Sigma^{\times g}/S_g$ is the space of unordered $g$--tuples of points in $\Sigma_g$, endowed with a nearly symmetric almost-complex structure, and
	\item differential is given by counting certain holomorphic maps of the disk into $Sym^g(\Sigma_g)$.
\end{itemize}

Now, if we have the data of a framed link $(L = K_1 \cup \ldots \cup K_n) \subset Y$, we can construct a {\em Heegaard multi-diagram} compatible with $L$ (and a corresponding filtered chain complex) by choosing
\begin{itemize}
	\item a {\em bouquet}, $B_L$, for $L$ \cite[Defn. 4.1]{MR2222356} and
	\item a self-indexing Morse function for $Y-N(B_L)$ with a single index $0$ critical point and no index $3$ critical points.
\end{itemize}

This specifies an ``unbalanced" Heegaard decomposition of $Y-N(B_L)$ which can be completed to a balanced Heegaard decomposition (and corresponding diagram) for any closed $3$--manifold obtained by Dehn filling $L$, just by appending the appropriate list of Dehn-filling curves $\beta_{g-n+1}, \ldots, \beta_g$ to the list $\{\beta_1, \ldots, \beta_{g-n}\}$ specifying $Y-N(B_L)$.  A Heegaard multi-diagram $\left(\Sigma_g, \boldalpha, \boldbeta_\mathcal{D},z\right)$ compatible with a set $\mathcal{D}$ of Dehn fillings of a link, $L \subset Y$, then simultaneously encodes the data of a Heegaard diagram for all Dehn fillings in $\mathcal{D}$.  In other words, one specifies a $g$--tuple, $\boldbeta_d$, of $\beta$ curves for each choice $d \in \mathcal{D}$, with the stipulation that $\{\beta_1, \ldots, \beta_{g-n}\}$ agree for all $\boldbeta_d$, as described above.  In the case of interest to us, $Y= \boldSigma(S^3, \bL)$, $N(L) = \bigcup_{i=1}^n\widetilde{N(C_i)}$, and $\mathcal{D} = \{0,1\}^n$.  The corresponding filtered chain complex has

\begin{itemize}
	\item generators given by the intersection points in $\bigcup_{d \in \mathcal{D}} \mathbb{T}_\alpha \cap \mathbb{T}_{\beta_d} $
	\item differential given by counting certain holomorphic polygons in $Sym^g(\Sigma_g)$, where
\begin{itemize}
	\item $2$--gons yield maps between generators associated to the same vertex of the cube of resolutions,
	\item $3$--gons (triangles) yield maps along edges of the cube, and
	\item $(k+2)$--gons yield maps along $k$--dimensional faces of the cube.
\end{itemize}
\end{itemize}

Please see \cite{MR2141852} for more details.

The Floer-theoretic filtered chain complex described above gives rise to a spectral sequence whose $E^2$ term is $\widetilde{Kh}(\overline{\bL})$ and whose $E^\infty$ term is $\widehat{HF}(\boldSigma(S^3, \bL))$.
\end{proof}

In \cite{GT07060741}, L. Roberts proved a refinement of the Ozsv{\'a}th-\Szabo spectral sequence for links in the solid torus complement of a standardly-imbedded unknot \[U = (z-\mbox{axis} \cup \infty) \subset  (S^3 = \mathbb{R}^3 \cup \infty).\]  Noting that this solid torus complement can be identified with the product sutured manifold $A \times I$ imbedded in $S^3$ as
\[A \times I := \{(r,\theta,z)\,\,|\,\,r \in [1,2], \theta \in (-\pi,\pi], z \in [0,1]\},\] we reinterpreted his result  in \cite{AnnularLinks} using the language of sutured manifolds, as follows:

\begin{theorem} \cite[Prop. 1.1]{GT07060741}, \cite[Thm. 2.1]{AnnularLinks} \label{thm:Annular} Let $\bL \subset A \times I$ be a link in the product sutured manifold $A \times I$.  Then there is a spectral sequence whose $E^2$ term is $Kh^*(\overline{\bL})$ and whose $E^\infty$ term is $SFH(\boldSigma(A \times I, \bL))$.
\end{theorem}

Here, $Kh^*$ is a version of Khovanov homology for links in thickened surfaces defined by Asaeda-Przytycki-Sikora in \cite{MR2113902}, and $SFH$ is the version of Heegaard Floer homology for balanced, sutured manifolds defined by \Juhasz in \cite{MR2253454}.  The proof of Theorem \ref{thm:Annular} follows Ozsv{\'a}th-Szab{\'o}'s original proof closely.  In fact, one sees that the chain complex computing $Kh^*(\bL \subset A \times I)$ (resp., $SFH(\boldSigma(A \times I, \bL))$) is really just the original chain complex computing the ordinary Khovanov homology of the link $Kh(\bL \subset S^3)$ (resp., a two--pointed version of the ordinary Floer homology of the double--branched cover, $\widehat{HF}(\boldSigma(S^3, \bL)) \otimes \mathbb{F}^2$), where the chain complex has been equipped with an extra $\mathbb{Z}$-- (resp., $\frac{1}{2}\mathbb{Z}$--) grading, and those components of the differential that do not respect this extra grading are set to zero.

There is a nice correspondence, detailed in Theorem \ref{thm:Dictionary}, between the extra Heegaard-Floer $\frac{1}{2}\Z$--grading, which we denote by ${\bf A}_S$, and the extra Khovanov $\Z$--grading, which we denote by $k$.  Furthermore, both gradings have natural interpretations, in terms of the geometry of vector fields on the Heegaard-Floer side (Section \ref{sec:SFH}) and in terms of the representation theory of the quantum group $U_{q}(\mathfrak{sl}_2)$ on the Khovanov side (Section \ref{sec:Kh}).  This correspondence will be the focus of the remainder of the paper.

\section{Sutured Floer grading} \label{sec:SFH}
We begin by studying the $\frac{1}{2}\Z$--grading on the sutured Floer homology of $\boldSigma(A \times I, \bL)$, as described in \cite{AnnularLinks}, following \cite{GT07060741}.

In particular, we build a filtered chain complex whose homology is $SFH(\boldSigma(A \times I, \bL))$ from the data of a sutured Heegaard multi-diagram for $\boldSigma(A \times I, \bL)$ and endow each generator, {\bf x}, of the complex with a $\frac{1}{2}\Z$--grading, ${\bf A}_S({\bf x})$.  We furthermore describe how to understand this grading as an Euler number of a natural vector field associated to ${\bf x}$.

Recall that any sutured Floer chain complex $CFH(Y,\Gamma)$ for a balanced sutured manifold $(Y,\Gamma)$ splits according to Spin$^c$ structures.  These are homology classes of non-vanishing vector fields on $Y$ with prescribed boundary behavior, and they form an affine set for the action of $H_1(Y;\Z)$.  Given a generator {\bf x} of a sutured Floer complex $CFH(Y,\Gamma)$, one assigns a Spin$^c$ structure to it as described in \cite[Defn. 4.5]{MR2253454}.  See \cite{MR1484699}, \cite[Sec. 3.2]{GT0512286}, \cite[Sec. 4]{MR2253454}, \cite[Sec. 3]{MR2390347} for more details.

Given the data of a trivialization, $t$, of the restriction of a Spin$^c$ structure, $\spincs$, to $\partial Y$ one can then define the first Chern class, $c_1(\spincs,t) \in H^2(Y,\partial Y;\Z)$, with respect to the trivialization (\cite[Defn. 3.7]{MR2390347}) as a relative Euler class of $\spincs$.  The added data of a properly-imbedded surface $(S,\partial S) \subset (Y,\partial Y)$ then endows $CFH(Y,\Gamma)$ with a $\frac{1}{2}\Z$ {\em Alexander$_S$--grading} ({\bf A}$_S$--grading) as described in \cite[Sec. 3]{SurfDecomp}, following \cite{MR2253454}.

Namely, given ${\bf x} \in CFH(Y,\Gamma)$, we define:
\[{\bf A}_S({\bf x}) := \frac{1}{2}\langle c_1(\spincs({\bf x}),t),[S]\rangle.\]

Furthermore, suppose that $(S,\partial S) \subset (Y,\partial Y)$ is geometrically disjoint from a framed link $L \subset Y$.  Then it is proved in \cite[Sec. 3]{SurfDecomp} that the entire filtered complex for $SFH(Y)$ associated to a Heegaard multi-diagram for the pair $(Y,L)$ (as described in Section 2) can be endowed with {\bf A}$_S$--gradings.  Moreover \cite[Lem. 3.12]{SurfDecomp} implies that the filtered chain complex splits according to these gradings.

The case of interest to us is when $S \subset \boldSigma(A \times I, \bL)$ is the preimage of a meridional disk in the solid torus $A \times I$ under the covering map $\pi: \boldSigma(A \times I, \bL) \rightarrow A \times I$.  In order to study these {\bf A}$_S$--gradings most easily, we construct a particularly convenient Heegaard multi-diagram which yields a filtered chain complex computing $SFH(\boldSigma(A \times I, \bL))$:

\begin{construction}\label{const:Heegmultidiag}
Identifying \[A \times I :=  \{(r,\theta,z)\,\,|\,\, r \in [1,2], \theta \in (-\pi,\pi], z \in [0,1]\} \subset \R^2 \times \R,\] let $\gamma_\phi \subset A$ denote the arc $\{(r,\theta) \in \R^2\,\,|\,\,\theta = \phi, r \in [1,2]\}$ for fixed $\phi \in (-\pi,\pi]$, $\gamma_{[\phi_1, \phi_2]} \subseteq A$ denote the closed region $\{(r,\theta) \in \R^2\,\,|\,\,\theta \in [\phi_1,\phi_2], r \in [1,2]\}$, and $w(\bL)$ denote the minimal geometric intersection number, $|\bL \pitchfork (\gamma_0 \times I)|$, among all elements in the isotopy class of $\bL$.

Choose an isotopy class representative of $\bL$ satisfying:
\begin{itemize}
	\item $\mathcal{P}(\bL)$ is an admissible projection (diagram) of $\bL$ to the annulus in the sense of \cite[Defn. 2.10]{AnnularLinks}, with small neighborhoods of the crossings $c_1, \ldots, c_n$ labeled $N(c_1), \ldots, N(c_n)$,
	\item the intersection of $\mathcal{P}(\bL)$ with $\gamma_{\left[-\frac{\pi}{4}, \frac{\pi}{4}\right]}$ in $A$ is ``standard" (see Figure \ref{fig:Standard}),
	\item the complete list, $\vec{p}$, of intersection points of $\bL$ with the half-level annulus $A \times \left\{\frac{1}{2}\right\}$ consists of
		\begin{itemize}
			\item the $w(\bL)$ points along $\gamma_0 \times \left\{\frac{1}{2}\right\}$,
			\item the $w(\bL)$ points along $\gamma_{\frac{\pi}{4}} \times \left\{\frac{1}{2}\right\}$, and
			\item the four points $\left(\mathcal{P}(\bL) \cap \partial(N(c_i))\right) \times \left\{\frac{1}{2}\right\}$ around each crossing for each $i \in 1, \ldots, n$,
		\end{itemize}
	\item $\bL$ intersects $A \times \left[\frac{1}{2}, 1\right]$ in (and only in) the regions $\gamma_{\left[0,\frac{\pi}{4}\right]} \times \left[\frac{1}{2},1\right]$ and $N(c_i) \times \left[\frac{1}{2},1\right]$.
\end{itemize}

\begin{figure}
\begin{center}
\resizebox{2in}{!}{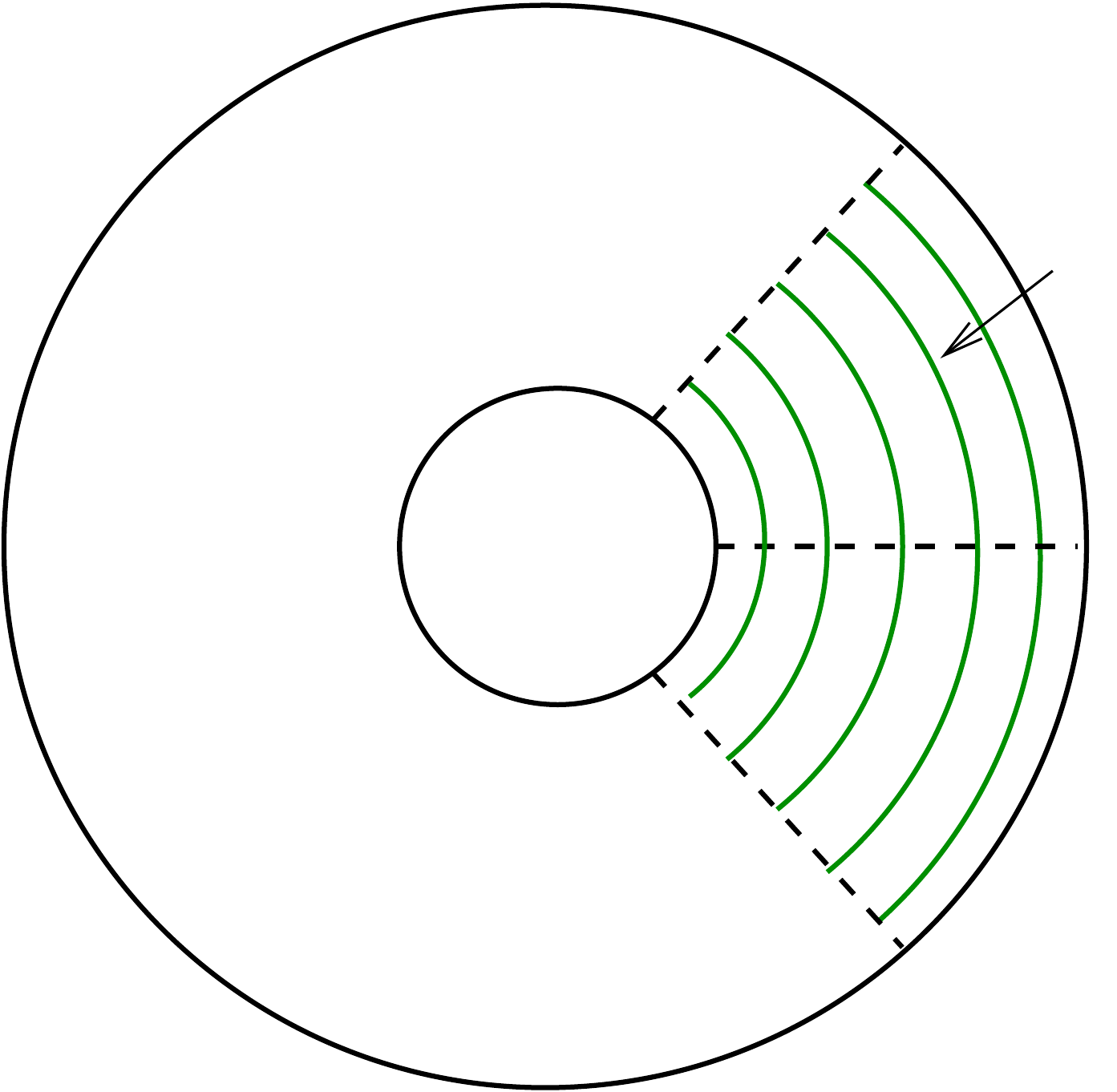}
\end{center}
\caption{A ``standard" projection, $\cP(\bL)$,  of $\bL$ to $\gamma_{\left[-\frac{\pi}{4},\frac{\pi}{4}\right]}$}
\label{fig:Standard}
\end{figure}

Note that if $\bL$ is as above, then every complete resolution of the corresponding projection, $\mathcal{P}(\bL)$, represents a resolved link that is a union of simple ``cups" and ``caps" as in \cite[Sec. 5.3]{MR2601010}.  We can therefore use the argument given in
 \cite[Proof of Prop. 2.8]{AnnularLinks} to construct a (not-yet-admissible) sutured Heegaard multi-diagram from which we can build the filtered chain complex computing $SFH(\boldSigma(A \times I, \bL))$.

More precisely, recall that the set $\{0,1\}^n$ indexes all possible complete resolutions of $\cP(\bL)$ (see Section \ref{sec:background}).  Let $\vec{p} \subset \left(A \times \left\{\frac{1}{2}\right\}\right)$ denote the set of intersection points (described above) of $\bL$ with $A \times \left\{\frac{1}{2}\right\}$.  Then the sutured Heegaard multi-diagram is given by $\left(\Sigma,\boldalpha,\boldbeta_{\{0,1\}^n}\right)$, where $\boldbeta_{\{0,1\}^n}$ represents the list of elements of the set $\{\boldbeta_{\cI}\,|\, \cI \in \{0,1\}^n\}$ and
\begin{itemize}
	\item $\Sigma := \boldSigma\left(A \times \{\frac{1}{2}\}, \vec{p}\right)$ (with covering projection $\pi: \boldSigma\left(A \times \{\frac{1}{2}\}, \vec{p}\right) \rightarrow A$),
	\item $\boldalpha := \pi^{-1}\left(\mathcal{P}(\mbox{cups})\right)$, and
	\item $\boldbeta_\cI := \pi^{-1}\left(\mathcal{P}(\mbox{caps}_{\cI})\right)$.
\end{itemize}

See Figure \ref{fig:Example} for an example.
\end{construction}

\begin{figure}
\begin{center}
\resizebox{5.5in}{!}{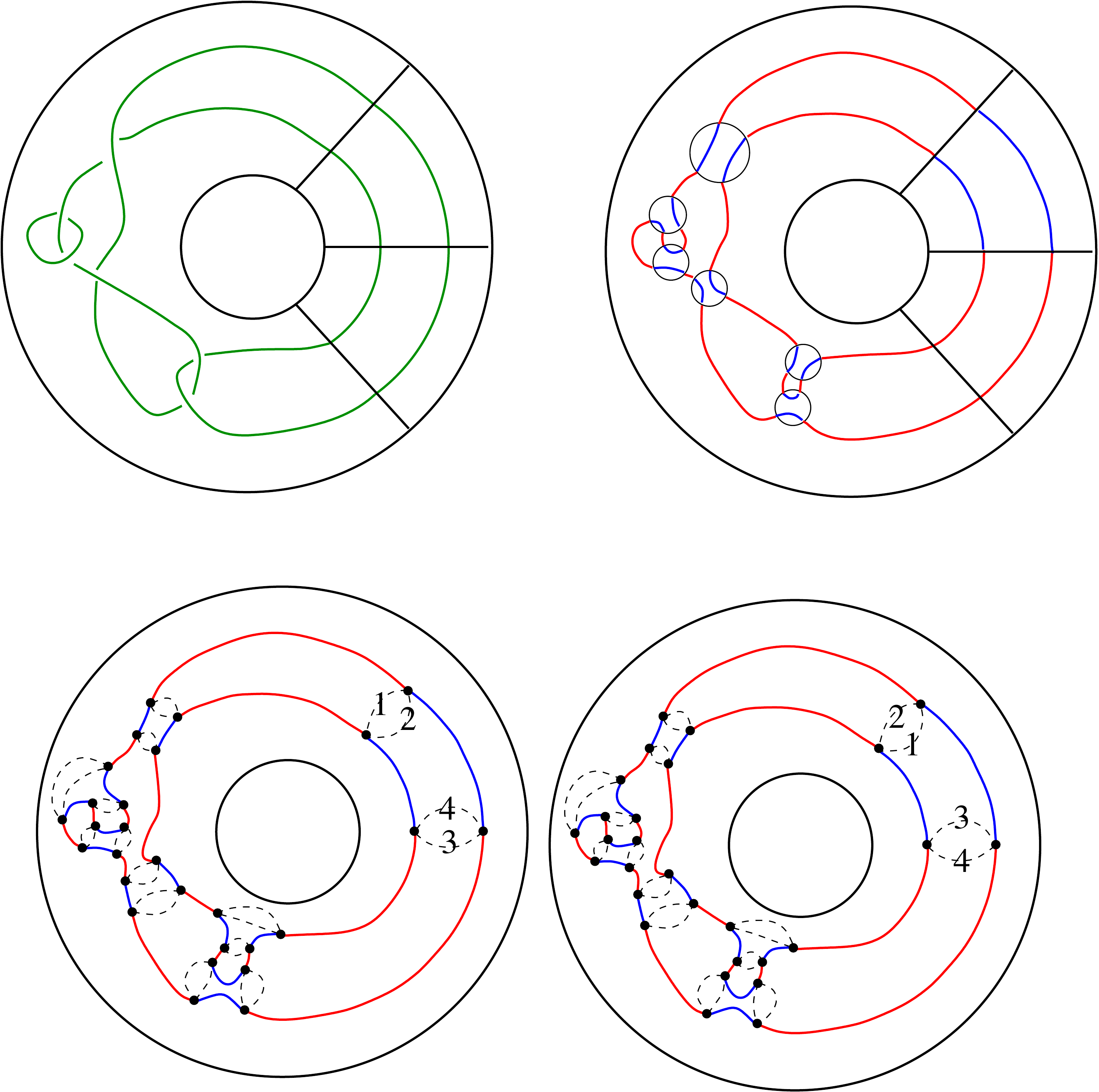}
\end{center}
\caption{The figure at top left is an admissible projection (diagram) of a link $\bL$ in the special form described in Construction \ref{const:Heegmultidiag}.  The figure at top right shows a diagram of a particular resolution, $\bL_{\cI}$, of $\bL$, with the regions corresponding to portions of $\bL$ in the region $A \times \left[\frac{1}{2},1\right]$ (the ``caps") marked in blue, the region $A \times \left[0,\frac{1}{2}\right]$ (the ``cups") marked in red, and the neighborhoods of the crossings indicated by circles.  The pair of figures in the second row above represent the two sheets of the double-branched cover of the resolution at top right, branched along the points of intersection of the resolved link with $A \times \left\{\frac{1}{2}\right\}$ (denoted by the bold black dots).  This figure will therefore yield a Heegaard diagram for $\boldSigma(A \times I, \bL_{\cI})$.  The dotted lines are branch cuts, which are identified on the two sheets in the standard way, as indicated by the numbers assigned to the pair of branch cuts on the right.  The red curves are the $\alpha$ circles, and the blue curves are the $\beta$ circles.}
\label{fig:Example}
\end{figure}

Let $\bL_{\cI}$ denote the complete resolution of $\bL$ associated to the $n$--tuple $\cI \in \{0,1\}^n$.  The construction of the sutured multi-diagram $\left(\Sigma, \boldalpha, \boldbeta_{\{0,1\}^n}\right)$ for $\boldSigma(A \times I, \bL)$ given in Construction \ref{const:Heegmultidiag} is tailored to have nice properties with respect to $\{S_{\cI}\}$, a particular collection of properly imbedded surfaces in the collection of sutured manifolds, $\{\boldSigma(A \times I, \bL_{\cI})\}$.  To explain what we mean, we must first recall a couple of definitions and facts.

Recall that if $F \subseteq (Y,\Gamma)$ is a {\em decomposing surface} \cite[Defn. 3.1]{MR723813} in a balanced, sutured manifold satisfying some additional mild hypotheses, then \Juhasz proves, in \cite[Prop. 4.4]{MR2390347}, that one can construct a balanced sutured Heegaard diagram $\Sigma$ {\em adapted} to $F$ (\cite[Defn. 4.3]{MR2390347}).  The key feature of such a Heegaard diagram is that ``most of" $F$ sits as a subsurface of $\Sigma$ with boundary a graph having certain important properties with respect to the curves $\boldalpha,\boldbeta \subset \Sigma$.  See Figure \ref{fig:SurfDiag} for an illustration, and \cite[Sec. 4]{MR2390347} for more details.

\begin{figure}
\begin{center}
\resizebox{3in}{!}{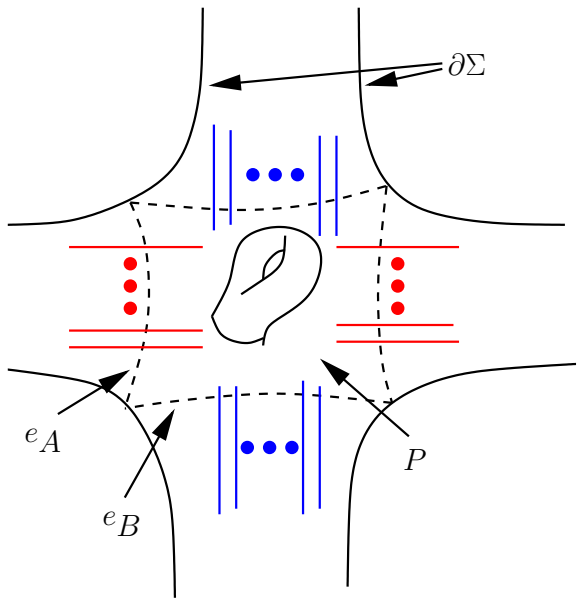}
\end{center}
\caption{An illustration of a Heegaard diagram, $\Sigma$, adapted to a decomposing surface, $F$.  The diagram $\Sigma$ contains a distinguished subsurface, $P$, whose boundary is a disjoint union of finitely many cyclic graphs with the property that all vertices lie on $\partial\Sigma$ and the edges alternate between edges $\{e_A\}$ of type ``A" and edges $\{e_B\}$ of type ``B."  The ``A" edges have trivial intersection with $\boldbeta$ and the ``B" edges have trivial intersection with $\boldalpha$.  The decomposing surface $F$ is then the surface (with corners) given by \[F = P \bigcup \left(\{e_A\} \times \left[\frac{1}{2}, 1\right]\right) \bigcup \left(\{e_B\} \times \left[0, \frac{1}{2}\right]\right).\]}
\label{fig:SurfDiag}
\end{figure}

In \cite[Prop. 4.11]{SurfDecomp}, we generalized this to the case where $F \subseteq (Y,\Gamma)$ is a decomposing surface (again, satisfying some additional mild hypotheses) that is geometrically disjoint from a framed link $L \subset (Y, \Gamma)$.  Letting $(Y_{\cI},\Gamma_{\cI})$ denote the sutured manifold obtained by performing $\cI$--surgery along the link, we then see (\cite[Defn. 3.10]{SurfDecomp}) that $F$ induces a compatible collection $\{F_{\cI} \subset Y_{\cI}\}$ of decomposing surfaces in the various surgeries on $L$.  We can analogously construct (\cite[Prop. 4.11]{SurfDecomp}) a sutured multi-diagram adapted to the collection $\{F_{\cI} \subset Y_{\cI}\}$  in the sense of \cite[Defn. 4.6]{SurfDecomp}.

In our situation, let $D \subset (A \times \left\{\frac{1}{2}\right\})$ be the bigon (quasi-polygon), pictured in Figure \ref{fig:Bigon}, with boundary graph given by the two labeled edges $e_A, e_B$.  Note that, by the way we chose $\bL$ in Construction \ref{const:Heegmultidiag}, $e_A$ only intersects the portion of $\cP(\bL)$ projecting from $\bL \cap (A \times \left[0,\frac{1}{2}\right])$ and $e_B$ only intersects the portion of $\cP(\bL)$ projecting from $\bL \cap (A \times \left[\frac{1}{2},1\right])$.  We now define the following piecewise smooth imbedded surface:
\[D' := D \cup \left(e_A \times \left[\frac{1}{2},1\right]\right) \cup \left(e_B \times \left[0,\frac{1}{2}\right]\right) \subset A \times I.\]

\begin{figure}
\begin{center}
\resizebox{3in}{!}{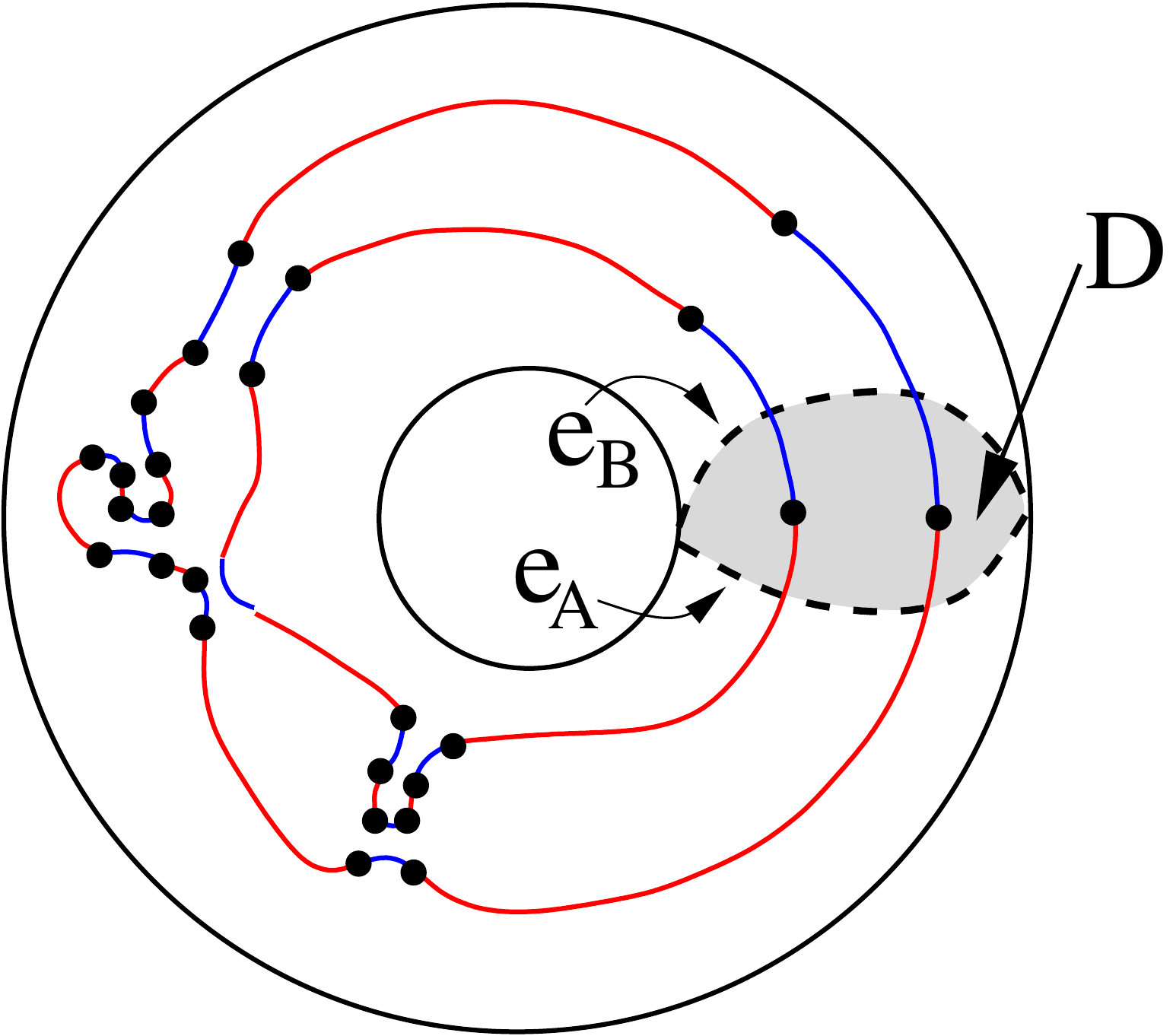}
\end{center}
\caption{An illustration of the bigon $D \subset (A \times \left\{\frac{1}{2}\right\})$ that gives rise to the (compatible collection of) decomposing surface(s) in a Heegaard multi-diagram for $\boldSigma(A \times I, \bL)$.  As before, the black dots represent intersection points of the (resolved) link with $A \times \left\{\frac{1}{2}\right\}$, and the red (resp., blue) arcs are the projections of the portions of the resolution in the bottom (resp., top) half of $A \times I$.}
\label{fig:Bigon}
\end{figure}

Letting
\begin{itemize}
	\item $\vec{q}$ denote the collection of intersection points $\vec{q} = (\bL \cap D)$ of the special isotopy class representative $\bL$ with $D \subset A \times \left\{\frac{1}{2}\right\}$ (which, by construction, occur along $\gamma_0 \times \left\{\frac{1}{2}\right\}$),
	\item and $\widetilde{e}_A, \widetilde{e}_B$ denote the lifts of $e_A, e_B$ to $\boldSigma(D, \vec{p})$,
\end{itemize}

we see that $D'$ lifts, in the double-branched cover $\boldSigma(A \times I, \bL)$, to a piecewise smooth surface, \[S' := \boldSigma(D, \vec{q}) \cup \left(\widetilde{e}_A \times \left[\frac{1}{2},1\right]\right) \cup \left(\widetilde{e}_B \times \left[0,\frac{1}{2}\right]\right).\]

In addition, the framed link $L_{\bL} \subset \boldSigma(A \times I, \bL)$ obtained as the lift of simple arcs associated to each of the crossings of $\cP(\bL)$ (as described in Section \ref{sec:background}) is geometrically disjoint from $S'$.  Hence we have a compatible collection $\{S'_{\cI} \subset \boldSigma(A \times I, \bL_{\cI})\}$ of surfaces associated to each choice $\cI \in \{0,1\}^n$ of surgery on $L_\bL$.  Furthermore, each $S'_{\cI}$ is isotopic to the preimage, in $\boldSigma(A \times I, \bL_{\cI})$, of the meridional disk $\gamma_0 \times I \subset A \times I$.  We see immediately:

\begin{lemma} Let $\bL \subset A \times I$ be a link.  A Heegaard multi-diagram $\left(\Sigma, \boldalpha, \boldbeta_{\{0,1\}^n}\right)$ associated to $\boldSigma(A \times I, \bL)$ via Construction \ref{const:Heegmultidiag} with the additional data of the subsurface $P= \boldSigma(D,\vec{q}) \subseteq \Sigma$ described above is a surface multi-diagram adapted to the collection $S'_{\cI} \subset \boldSigma(A \times I, \bL_{\cI})$.
\end{lemma}

\begin{proof}  This is essentially by definition; see \cite[Defn. 4.6]{SurfDecomp}.  Note that the boundary of $\boldSigma(D,\vec{q})$ decomposes as a union of edges $\widetilde{e}_A$ and $\widetilde{e}_B$.  Since (downstairs in $D$) $e_A$ only intersects the portions of the projection $\cP(\bL)$ coming from $\bL \cap (A \times \left[0, \frac{1}{2}\right])$ and $e_B$ only intersects the portions of $\cP(\bL)$ coming from $\bL \cap (A \times \left[\frac{1}{2},1\right])$, the lifts, $\widetilde{e}_A, \widetilde{e}_B$ will satisfy the necessary condition
\[ \widetilde{e}_A \cap \boldbeta_\cI = \emptyset, \,\, \widetilde{e}_B \cap \boldalpha = \emptyset\] for all $\cI \in \{0,1\}^n$.

Any isotopies needed to make $\left(\Sigma, \boldalpha, \boldbeta_{\{0,1\}^n}\right)$ admissible (see \cite[Lem. 3.13]{MR2601010}) can be chosen with support away from $\boldSigma(D, \vec{q})$.
\end{proof}

Since the Heegaard multi-diagram $\left(\Sigma, \boldalpha, \boldbeta_{\{0,1\}^n}\right)$ is adapted to the collection $S_\cI'$, it is easy to compute the $\frac{1}{2}\Z$-- Alexander grading associated to any generator of the filtered chain complex for an appropriate trivialization, $t$.  Namely, a typo-free version of the formula given in the proof of \cite[Thm. 2.20]{GT08110178} (see also \cite{MR2065507, MR2390347}), adjusted to our situation (our decomposing surface intersects the sutures in $4$ points, not $2$) tells us that, for some trivialization, $t$, \[{\bf A}_{S}({\bf x}) = \frac{1}{2}\left(\chi(S) - 2 + 2\cdot\#\{x_i \in {\bf x}\,|\, x_i \in P\}\right).\]

Noting that $\chi(S)= \chi(\boldSigma(D,\vec{q})) = 2 - |\vec{q}|$ by the Riemann-Hurwitz formula, we obtain:
\[{\bf A}_{S}({\bf x}) = -\frac{1}{2}(|\vec{q}|) + \#\{x_i \in {\bf x}\,|\, x_i \in P\}.\]

Since $\#\{x_i \in {\bf x}\,|\, x_i \in P\} \in \{0, \ldots, |\vec{q}|\}$, we see that the Alexander gradings of generators range from $-\frac{1}{2}(|\vec{q}|)$ (when there are no intersection points in $P$) to $\frac{1}{2}(|\vec{q}|)$ (when there are the maximal number of intersection points in $P$).

In summary, the $\frac{1}{2}\Z$ grading, ${\bf A}_{S}({\bf x})$, assigned to a generator, {\bf x}, is the Euler number of a vector field on $\boldSigma(A \times I, \bL_{\cI})$ naturally defined using the data of {\bf x}.  Furthermore, this Euler number can be computed explicitly by counting the number of intersection points of {\bf x} that lie in the region, $P \subset \Sigma$, corresponding to the family of decomposing surfaces, $\{S_\cI\}$.

\section{Khovanov grading} \label{sec:Kh}
\subsection{Khovanov homology for links in thickened annuli}
As before, let $\mathbb{L} \subset A \times I \subset  S^3$ be a link in the product sutured manifold $A \times I$, where $A \times I$ has been imbedded in $S^3$ via the imbedding described before Theorem \ref{thm:Annular}, and let $\mathcal{P}(\mathbb{L})\subset A \subset \mathbb{R}^2$ be a diagram for $\mathbb{L}$ with $n$ crossings.  We similarly study a $\Z$--grading on the Khovanov homology of $\bL$, as described in \cite{MR2113902} (see also \cite{GT07060741}, \cite{AnnularLinks}).

If we forget about the imbedding $\bL \subset A \times I$ and view $\bL$ as a link in $S^3$, we can define its ordinary (nonreduced) Khovanov homology in a manner very similar to the definition of the reduced Khovanov homology summarized during the proof of Theorem \ref{thm:OzSz}.  In particular, the Khovanov homology of $\mathbb{L}$ is defined by associating to each complete resolution $\mathbb{L}_{\mathcal{I}}$ a graded vector space $V^{\otimes c}$, where $c$ is the number of components in $\mathbb{L}_{\mathcal{I}}$, and $V$ is a two-dimensional graded vector space spanned by two homogeneous elements $v_-$ and $v_+$ of degrees $-1$ and $+1$, respectively. In fact, it is understood implicitly that we have chosen a numbering of the components of $\mathbb{L}_{\mathcal{I}}$ from $1$ to $c$, so that the $i$th tensor factor of $V^{\otimes c}$ corresponds naturally to the $i$th component of $\mathbb{L}_{\mathcal{I}}$.

Let $\mathbb{L}_{\mathcal{I}}$ be a complete resolution of $\mathbb{L}$ with components $K_1,\ldots,K_c$ as above.  If we now remember the extra data of the imbedding $\mathbb{L} \subset A \times I \subset S^3$ (correspondingly, $\mathcal{P}(\mathbb{L}) \subset A \subset \R^2$), we can endow the vector space $V^{\otimes c}$ associated to $\mathbb{L}_{\mathcal{I}}$ with an additional grading, called the $k-$grading. Namely, we define the $k-$degree of a basis vector $v_{\epsilon_1}\otimes\ldots\otimes v_{\epsilon_c}\in V^{\otimes c}$, for $\epsilon_i\in\{+,-\}$, by $k(v_{\epsilon_1}\otimes\ldots\otimes v_{\epsilon_m}):=k_1(v_{\epsilon_1})+\ldots+k_m(v_{\epsilon_m})$, where
$$
k_i(v_{\pm}):=\begin{cases}
0&\text{if $[\mathcal{P}(K_i)]=0\in H_1(A;\mathbb{Z})$,}\\
\pm 1&\text{if $[\mathcal{P}(K_i)]\neq 0\in H_1(A;\mathbb{Z})$.}
\end{cases}
$$

It turns out that the differential in Khovanov's chain complex does not preserve the $k-$grading; however, it is non-increasing in the $k$--grading, which induces a filtration on Khovanov's chain complex.  By forming the associated graded chain complex and taking its homology, one obtains an abelian group $Kh^*(\mathbb{L})$.   This group is (up to an overall grading shift) the version of Khovanov homology for links in $A\times I$ that was introduced by Asaeda--Przytycki--Sikora \cite{MR2113902}. $Kh^*(\mathbb{L})$ is equipped with three gradings: two gradings coming from the ordinary bigrading on Khovanov's chain complex (in \cite{MR1740682}, these gradings were denoted by $i$ and $j$), and the $k-$grading.

The graded Euler characteristic of $Kh^*(\mathbb{L})$ is the Laurent polynomial
$$
\chi_{q,t}(Kh^*(\mathbb{L})):=\sum_{i,j,k}(-1)^iq^jt^k \mbox{dim}_{\mathbb{F}}(Kh^*(\mathbb{L})^{i,j,k})\,\in\,\mathbb{Z}[q^{\pm 1},t^{\pm 1}]\,,
$$
where $Kh^*(\mathbb{L})^{i,j,k}$ denotes the homogeneous component of $Kh^*(\mathbb{L})$ for the degrees $(i,j,k)$. We will henceforth abbreviate $SJ(\mathbb{L}):=\chi_{q,t}(Kh^*(\mathbb{L}))$. Note that $SJ(\mathbb{L})$ specializes to the Jones polynomial $J(\mathbb{L})$ if one sets $t$ equal to $1$. If $\mathbb{L}_{\mathcal{I}}$ is a complete resolution with $c$ components, then $SJ(\mathbb{L}_{\mathcal{I}})$ is given by
\begin{equation}\label{eqn:sj}
SJ(\mathbb{L}_{\mathcal{I}}):=\chi_{q,t}(V^{\otimes c})=
(q+q^{-1})^{u}\, z^{c-u}\,,
\end{equation}
where $z:=qt+(qt)^{-1}$, and $u$ is the number of components of $\mathbb{L}_{\mathcal{I}}$ that are nullhomologous in $A\times I$. In particular, since $SJ(\mathbb{L})$ is a linear combination of the $SJ(\mathbb{L}_{\mathcal{I}})$ with coefficients in $\mathbb{Z}[q^{\pm 1}]^{\times}$, this shows that $SJ(\mathbb{L})$ is actually contained in the subring $\mathbb{Z}[q^{\pm 1}][z]\subset\mathbb{Z}[q^{\pm 1},t^{\pm 1}]$. In Subsection \ref{sec:Skein}, we will discuss how this subring is related to the Kauffman bracket skein module of $A\times I$.

\subsection{Enhanced resolutions} \label{sec:kgrading}
Let $\mathbb{L}$ be a link in $A\times I$ with diagram $\cP(\bL)$, as before.  An enhanced resolution of $\cP(\mathbb{L})$ is an oriented $1-$manifold, $\mathbb{S}\subset A$, whose underlying unoriented $1-$manifold coincides with the projection, $\cP(\bL_{\cI})$, of a complete resolution of $\mathbb{L}$. (In \cite{viro-2002}, the term ``enhanced Kauffman state'' was used with a similar meaning).

To obtain a more geometric interpretation of the $k-$grading, we will now reinterpret the generators of $Kh^*(\mathbb{L})$ in terms of enhanced resolutions. Namely, if
$V^{\otimes c}$ denotes the vector space associated to a complete resolution $\mathbb{L}_{\mathcal{I}}$, then we will identify the basis vector $v_{\epsilon_1}\otimes\ldots\otimes v_{\epsilon_c}\in V^{\otimes c}$ ($\epsilon_i\in\{+,-\}$) with the enhanced resolution $\mathbb{S}\subset A$ which is obtained by orienting the $i$th component of $\mathcal{P}(\mathbb{L}_{\mathcal{I}})\subset A\subset\mathbb{R}^2$ clockwise if $\epsilon_i=-$, and counterclockwise if $\epsilon_i=+$.



Using this correspondence between basis vectors and enhanced resolutions, we can now express the $k-$grading as an algebraic intersection number
$$
k(v_{\epsilon_1}\otimes\ldots\otimes v_{\epsilon_c})=[\gamma_{0}]\cdot [\mathbb{S}]\,,
$$
where $\mathbb{S}$ is the enhanced resolution corresponding to the basis vector $v_{\epsilon_1}\otimes\ldots\otimes v_{\epsilon_c}$, and $\gamma_{0}\subset A$ denotes the arc introduced in Construction \ref{const:Heegmultidiag}, oriented outward (in the direction of increasing $r$).  Moreover, Formula \eqref{eqn:sj} can be rewritten as
\begin{equation}\label{eqn:sjstate}
SJ(\mathbb{L}_{\mathcal{I}})=\sum_{\mathbb{S}\in E(\mathbb{L}_{\mathcal{I}})}q^{j(\mathbb{S})}t^{k(\mathbb{S})}\,,
\end{equation}
where $k(\mathbb{S}):=[\gamma_{0}]\cdot [\mathbb{S}]$, and $j(\mathbb{S})$ denotes the difference between the number of counterclockwise and clockwise circles in $\mathbb{S}$, and $E(\mathbb{L}_{\mathcal{I}})$ denotes the set of all enhanced resolutions of $\mathbb{L}$ whose underlying unoriented $1-$manifold is equal to $\mathcal{P}(\mathbb{L}_{\mathcal{I}})$. By summing over all complete resolutions $\mathbb{L}_{\mathcal{I}}$, one then obtains a formula expressing $SJ(\mathbb{L})$ as a sum over all enhanced resolutions of $\mathbb{L}$.

\subsection{Reshetikhin-Turaev invariant}  As mentioned in the introduction, the integer $k-$grading defined above has a natural representation-theoretic interpretation in terms of weight space decompositions for the Reshetikhin-Turaev invariant.

Again, let $\bL \subset A \times I \subset S^3$ be a link and $\cP(\bL)$ a diagram of $\bL$.  In the following, we will assume that the intersection of $\cP(\bL)$ with $\gamma_{\left[-\frac{\pi}{4},\frac{\pi}{4}\right]} \subset A$ is ``standard'', in the sense of Construction \ref{const:Heegmultidiag}.

Let $\mathbb{T}'$ be the tangle in $D^2 \times I$ obtained by removing the neighborhood of a meridional disk in the solid torus $A \times I$:
$$
\mathbb{T}':=\mathbb{L}\setminus \left(\gamma_{\left(-\frac{\pi}{4},\frac{\pi}{4}\right)}\times [0,1]\right)\, \subset (A \times I) \setminus \left(\gamma_{\left(-\frac{\pi}{4},\frac{\pi}{4}\right)}\times [0,1]\right).
$$

For convenience, we now change the angular $(\theta)$ coordinate on $(A \times I) \setminus \left(\gamma_{\left(-\frac{\pi}{4},\frac{\pi}{4}\right)}\times [0,1]\right)$ so that it is continuous on this region, by letting \[\vartheta = \left\{\begin{array}{rl}
						\pi - \theta & \mbox{if $\theta \in \left[\frac{\pi}{4}, \pi\right]$, and}\\
						-\pi-\theta & \mbox{if $\theta \in \left(-\pi,-\frac{\pi}{4}\right]$.}\end{array}\right.\]

This allows us to easily identify  $\mathbb{T}'$ with a tangle $\mathbb{T} \subset [1,2] \times \left[\frac{-3\pi}{4},\frac{3\pi}{4}\right] \times [0,1] \subset \R^3$ by sending a point \[(r,\vartheta,z)\in (A \times I) \setminus \left(\gamma_{\left(-\frac{\pi}{4},\frac{\pi}{4}\right)}\times [0,1]\right)\] to the point $(r,\vartheta,z)\in [1,2]\times \left[-\frac{3\pi}{4},\frac{3\pi}{4}\right]\times [0,1]$. If the ``standard" intersection of $\bL$ with the region $\gamma_{\left(-\frac{\pi}{4},\frac{\pi}{4}\right)} \times [0,1]$ consists of $m$ arcs, then $\mathbb{T}$ has $m$ top endpoints (with $\vartheta-$coordinate equal to $\frac{3\pi}{4}$) and $m$ bottom endpoints (with $\vartheta-$coordinate equal to $-\frac{3\pi}{4}$).

To $\mathbb{T}$, the Reshetikhin-Turaev construction \cite{reshetikhin-1990} for the quantum group $U_q(\mathfrak{sl}_2)$ associates a linear map
$$
J(\mathbb{T})\,\colon\, V_1^{\otimes m}\longrightarrow V_1^{\otimes m}
$$
which intertwines the quantum group action. Here, $V_1$ is the two-dimensional fundamental representation of $U_q(\mathfrak{sl}_2)$, with underlying vector space $V_1:=\mathbb{C}(q)v_1\oplus\mathbb{C}(q)v_{-1}$. The generators $E,F,K$ of $U_q(\mathfrak{sl}_2)$ act by
$$
Kv_1=qv_1\,,\qquad Kv_{-1}=q^{-1}v_{-1}\,,\qquad Ev_1=Fv_{-1}=0\,,\qquad
Ev_{-1}=v_1\,,\qquad Fv_{1}=v_{-1}\,.
$$

Note that when $m=0$ (i.e., if $\mathbb{T}$ is a link), $J(\mathbb{T})$ is given by scalar multiplication by the Jones polynomial $J(\mathbb{L})$. In general, $J(\mathbb{T})$ can be described by a state sum formula, as we will now explain. See \cite{khovanov-1997} for a reference.

As before, an enhanced resolution of $\mathbb{T}$ is an oriented $1-$manifold $\mathbb{S}\subset [1,2]\times \left[-\frac{3\pi}{4},\frac{3\pi}{4}\right]$ whose underlying unoriented $1-$manifold is the projection of a resolution $\mathbb{T}_{\mathcal{I}}$. Given an enhanced resolution $\mathbb{S}$ of $\mathbb{T}$, we can associate two sequences of up and down arrows: the bottom sequence $b(\mathbb{S})\in\{\uparrow,\downarrow\}^m$, obtained by reading off the local orientations of the arc components of $\mathbb{S}$ near the $m$ bottom endpoints; and the top sequence $t(\mathbb{S})\in\{\uparrow,\downarrow\}^m$, obtained by reading off the local orientations of the arc components of $\mathbb{S}$ near the $m$ top endpoints.

Each sequence ${\ba}\in\{\uparrow,\downarrow\}^m$ determines a standard basis vector in $V_1^{\otimes m}$, namely the vector obtained by replacing each $\uparrow$ by $v_1\in V_1$, and each $\downarrow$ by $v_{-1}\in V_1$. For example, the sequence $\ba=(\uparrow\uparrow\downarrow\uparrow\downarrow)\in\{\uparrow,\downarrow\}^5$ corresponds to the standard basis vector
$$
v_{\ba}:=v_1\otimes v_1\otimes v_{-1}\otimes v_1\otimes v_{-1}\,\in\,V_1^{\otimes 5}\,.
$$

The matrix entry $J(\mathbb{T}_{\mathcal{I}})_{\ba,\bb}:=v_{\ba}^*(J(\mathbb{T})(v_{\bb}))$ of the linear map $J(\mathbb{T}_{\mathcal{I}})\colon V_1^{\otimes m}\rightarrow V_1^{\otimes m}$ can now be defined explicitly by
\begin{equation}
\label{eqn:statesumtangle}
J(\mathbb{T}_{\mathcal{I}})_{\ba,\bb}:=
\sum_{
\begin{smallmatrix}
\mathbb{S}\in E(\mathbb{T}_{\mathcal{I}})\\
b(S)=\bb,\,
t(S)=\ba
\end{smallmatrix}
} q^{j(\mathbb{S})}\,,
\end{equation}
where $E(\mathbb{T}_{\mathcal{I}})$ is the set of all enhanced resolutions of $\mathbb{T}$ whose underlying unoriented $1-$manifold is equal to $\mathcal{P}(\mathbb{T}_{\mathcal{I}})$, and $j(\mathbb{S})$ is defined as follows. Let $G^{\mathbb{S}}\colon \mathbb{S}\rightarrow S^1$ be the Gauss map which sends a point $p\in \mathbb{S}\subset [1,2]\times \left[-\frac{3\pi}{4},\frac{3\pi}{4}\right]$ to the positive unit tangent vector of $\mathbb{S}$ at $p$. If we assume that $\mathbb{S}$ is vertical near its endpoints, then the image of the fundamental class of $\mathbb{S}$ defines a relative homology class $G^{\mathbb{S}}_*[\mathbb{S}]\in H_1(S^1,\{\pm i\})$, where $i=\sqrt{-1}\in S^1$. Now $j(\mathbb{S})$ is the algebraic intersection number $j(\mathbb{S}):=[\{1\}]\cdot G^{\mathbb{S}}_*[\mathbb{S}]$.

\subsection{Weight spaces}
As a vector space, $V_1^{\otimes m}$ decomposes into weight spaces, i.e. into eigenspaces for the action of $K\in U_q(\mathfrak{sl}_2)$:
$$
V_1^{\otimes m}=\bigoplus_{n=0}^m V_1^{\otimes m}[m-2n]
$$
Here, $V_1^{\otimes m}[\lambda]$ is the eigenspace for the eigenvalue $q^{\lambda}$ for the $K$-action, and the generators $E$ and $F$ act by maps $E_{\lambda}\colon V_1^{\otimes m}[\lambda]\rightarrow V_1^{\otimes m}[\lambda+2]$ and $F_{\lambda}\colon V_1^{\otimes m}[\lambda]\rightarrow V_1^{\otimes m}[\lambda-2]$, respectively.

Since $K$ acts on tensor products by $K(v\otimes w):=(Kv)\otimes (Kw)$, and since $v_1$ and $v_{-1}$ are eigenvectors for the $K$-action with eigenvalues $q^{\pm 1}$, each standard basis vector $v_{\ba}$ for $\ba\in\{\uparrow,\downarrow\}^m$, is an eigenvector for the $K$-action with
$$
Kv_{\ba}=q^{k(\ba)}v_{\ba}\,,
$$
where
$$
k(\ba)=\#(\mbox{up arrows in $\ba$})-\#(\mbox{down arrows $\ba$})\,.
$$
Thus, the weight space $V_1^{\otimes m}[\lambda]$ is the span of all $v_{\ba}$ with $k(\ba)=\lambda$. For the following discussion, it will be useful to observe that $k(\ba)$ can be identified with the algebraic intersection number $k(\ba)=[\ell]\cdot [\ba]$, where here the arrows in $\ba$ are viewed as upward and downward oriented vertical lines, and $\ell$ is a horizontal line, oriented to the right, and intersecting each of the arrows in $\ba$ in a single point.

\subsection{$SJ(\mathbb{L})$ and quantum trace}
If $\mathbb{S}$ is an enhanced resolution of $\mathbb{T}$ with $t(\mathbb{S})=\ba$ and $b(\mathbb{S})=\bb$ for $\ba,\bb\in\{\uparrow,\downarrow\}^m$, then $k(\ba)=[\ell]\cdot [\ba]=[\gamma_{\phi}]\cdot [\mathbb{S}]=[\ell]\cdot [\bb]=k(\bb)$, where $\gamma_{\phi}$ is the rightward-oriented horizontal line $\gamma_{\phi}:=[1,2]\times\{\phi\}$ for arbitrary $\phi\in \left[-\frac{3\pi}{4},\frac{3\pi}{4}\right]$. Hence it follows that $J(\mathbb{T})\colon V_1^{\otimes m}\rightarrow V_1^{\otimes m}$ preserves the weight space decomposition. (Alternatively, this follows from the fact that $J(\mathbb{T})$ intertwines the quantum group action).


Recalling that $[\gamma_{\phi}]\cdot [\mathbb{S}]=k(\mathbb{S})$ (where $k(\mathbb{S})$ is the $k-$degree defined immediately after equation \eqref{eqn:sjstate}), we see that the restriction of $J(\mathbb{T})$ to the weight space $V_1^{\otimes m}[k]$ is given by summing over enhanced resolutions with $k(\mathbb{S})=k$.

Now note that every enhanced resolution of $\mathbb{T}$ which satisfies $t(\mathbb{S})=b(\mathbb{S})$ can be lifted to an enhanced resolution of $\mathbb{L}$. Hence we have a bijection
$$
E(\mathbb{L})\,\stackrel{1:1}{\longleftrightarrow}\,\{\mathbb{S}\in E(\mathbb{T})\,\colon\,t(\mathbb{S})=b(\mathbb{S})\}\,.
$$
where $E(\mathbb{L})$ (resp., $E(\mathbb{T})$) denotes the set of all enhanced resolutions of $\mathbb{L}$ (resp., $\mathbb{T}$).

It is fairly straightforward to check that $j(\mathbb{S}')=j(\mathbb{S})+ k$ whenever $\mathbb{S}'$ is an enhanced resolution of $\mathbb{L}$ corresponding via the above bijection to an enhanced resolution $\mathbb{S}$ of $\mathbb{T}$ with $t(\mathbb{S})=b(\mathbb{S})$ and $k(\mathbb{S})=k$. Hence


$$
SJ(\mathbb{L}_{\mathcal{I}})=
\sum_{\mathbb{S}'\in E(\mathbb{L}_{\mathcal{I}})}q^{j(\mathbb{S}')}t^{k(\mathbb{S}')}
=
\sum_k (qt)^k
\sum_{
\begin{smallmatrix}
\mathbb{S}\in E(\mathbb{T}_{\mathcal{I}})\\
t(\mathbb{S})=b(\mathbb{S}),\,
k(\mathbb{S})=k
\end{smallmatrix}
}q^{j(\mathbb{S})}\,,
$$
and thus we obtain
$$
SJ(\mathbb{L})
=\sum_k (qt)^k\sum_{
\begin{smallmatrix}
\ba\in\{\uparrow,\downarrow\}^m\\
 k(\ba)=k
\end{smallmatrix}
}
J(\mathbb{T}_{\mathcal{I}})_{\ba,\ba}
=\sum_k (qt)^k\, \mbox{tr}(J(\mathbb{T})|_{V_1^{\otimes m}[k]})\,.
$$
where $\mbox{tr}(-)$ stands for the ordinary trace for $\mathbb{C}(q)$-linear endomorphisms. Setting $t$ equal to $1$, we recover the well-known fact that the Jones polynomial of $\mathbb{L}$ can be expressed as the quantum trace $\mbox{tr}_q(J(\mathbb{T}))$, defined via the action of the generator $K$ on  $V_1^{\otimes m}$ by $\mbox{tr}_q(J(\mathbb{T})):=\mbox{tr}(K\circ J(\mathbb{T}))$:
$$
J(\mathbb{L})
=\sum_k q^k\, \mbox{tr}(J(\mathbb{T})|_{V_1^{\otimes m}[k]})=\mbox{tr}_q(J(\mathbb{T}))\,.
$$

\subsection{Skein module interpretation of $SJ(\mathbb{L})$} \label{sec:Skein}
Let $a$ be a formal variable satisfying $q=-a^{-2}$, and let $\mathbb{L}_f$ denote the framed link $\mathbb{L}_f\subset A\times I$ obtained by equipping the (oriented) link $\mathbb{L}\subset A\times I$ with the framing induced by a compact oriented surface $F\subset A\times [0,1]$ with $\mathbb{L}\subset\partial F$ and $\partial F\setminus \mathbb{L}\subset (\partial A)\times [0,1]$.

Recall that the Kauffman bracket skein module of $A\times I$ is defined as the quotient \[K_a(A\times I):=M/R,\] where $M$ is the free $\mathbb{Z}[a^{\pm 1}]$-module generated by all isotopy classes of framed links in $A\times I$, and $R$ is the submodule of $M$ generated by all Kauffman bracket skein relations (cf. \cite{przytycki-1991}, \cite{turaev-1988}). Now it follows essentially from the definitions that $SJ(\mathbb{L})$ is equal to the element
$$
\varphi^{-1}([\mathbb{L}_f])\in \mathbb{Z}[q^{\pm 1}][z]\subset\mathbb{Z}[a^{\pm 1}][z]
$$
where $\varphi$ is the isomorphism
$$
\varphi\,\colon\,\mathbb{Z}[a^{\pm 1}][z]\stackrel{\cong}{\longrightarrow}K_a(A\times [0,1])
$$
given by sending $z^n$ to a collection of $n$ disjoint homologically nontrivial imbedded circles in $A\times\{1/2\}$, equipped with the framing induced by the surface $A\times\{1/2\}$.

\section{The Relationship between the Khovanov and Sutured Floer Gradings}

We have seen that if $\bL \subset A \times I$ is a link in the thickened annulus and $\cP(\bL) \subset A$ is a diagram for $\bL$  that has been placed in a ``standard" form via Construction \ref{const:Heegmultidiag}, we can construct from the data of $\cP(\bL)$ either:
\begin{enumerate}
	\item a filtered sutured Floer chain complex for $\boldSigma(A \times I, \bL)$, as described in Sections \ref{sec:background} and \ref{sec:SFH}, or
	\item a Khovanov chain complex for $\overline{\bL}$, the mirror of $\bL$, as described in Section \ref{sec:Kh}.
\end{enumerate}

Furthermore, there is a bijection between the generators of the two chain complexes (\cite[Prop. 6.2]{MR2141852}, \cite[Prop. 2.28]{AnnularLinks}).  Each generator can therefore be assigned either:
\begin{enumerate}
	\item a geometric $\frac{1}{2}\mathbb{Z}$ sutured Floer grading, called the ${\bf A}_S$ grading, or
	\item a representation-theoretic $\mathbb{Z}$ Khovanov grading, called the $k$--grading.
\end{enumerate}

We now show that these two gradings are closely related.  Namely:
\begin{theorem} \label{thm:Dictionary} Let $\bL \subset A \times I$ be a link in the product sutured manifold $A \times I$, $\overline{\bL}$ its mirror, and $\cP(\bL)$ (resp., $\cP(\overline{\bL})$) a diagram for $\bL$ (resp., $\overline{\bL}$) produced by Construction \ref{const:Heegmultidiag}.  Then if ${\bf x}$ is a generator of the associated (filtered) chain complex, we have:
\[k({\bf x}) = -2 {\bf A}_S({\bf x}).\]
\end{theorem}

\begin{proof} We begin by noting that the generators of both the Khovanov and sutured Floer complex are in one-to-one correspondence with enhanced Kauffman states.  On the Khovanov side, this correspondence is explained in Section \ref{sec:Kh}.

On the sutured Floer side, recall that we have one generator for each $d$--tuple of intersection points between $\boldalpha$ and $\boldbeta_\cI$, for each $\cI \in \{0,1\}^\ell$.  Furthermore, each connected component of the diagram, $\cP(\bL_\cI)$, of a resolution gives rise to a cyclic alternating chain of $\alpha$ and $\beta$ circles in the Heegaard diagram for $\boldSigma(A \times I, \bL_\cI)$, each intersecting the next in a single point.  Hence, once a single intersection point in a chain is chosen, the other intersection points on that chain are completely determined by the requirement that a $d$--tuple of intersection points yielding a Heegaard Floer generator must occupy each $\alpha$ circle (resp., each $\beta$ circle) exactly once.  There are then two choices associated to each such cyclic chain, corresponding to the two choices of orientation for that connected component in the relevant enhanced Kauffman state.

Now let $\mathbb{S}$ be an enhanced Kauffman state for $\cP(\bL)$ (resp., $\cP(\overline{\bL})$) and ${\bf a} \in \{\uparrow,\downarrow\}^m$ the associated sequence of $\uparrow$'s and $\downarrow$'s seen along $\gamma_0$, oriented outward.  Section \ref{sec:Kh} tells us that \[k(\mathbb{S}) = \#(\uparrow) - \#(\downarrow).\]

Now recall that $w(\bL) = |\vec{q}|$ denotes the number of arcs of $\cP(\bL)$ in the standard region $\gamma_{\left[-\frac{\pi}{4},\frac{\pi}{4}\right]} \subset A$.  On the sutured Floer side, we have $|\vec{q}|$ distinguished intersection points in the special subsurface, $P$, lying along the preimage of $\gamma_0$ in $\Sigma$.  Furthermore, the correspondence between enhanced Kauffman states and generators of the sutured Floer chain complex guarantees that an intersection point is {\em occupied} (resp., {\em unoccupied}) iff the corresponding arrow for the enhanced Kauffman state at that position is $\downarrow$ (resp., $\uparrow$).  Hence, the end of Section \ref{sec:SFH} tells us that
\begin{eqnarray*}
{\bf A}_S(\mathbb{S}) &=& -\frac{1}{2}(|\vec{q}|) + \#(\downarrow)\\
				  &=& \frac{1}{2}(- \#(\uparrow) + \#(\downarrow))\\
				  &=& -\frac{1}{2}(k(\mathbb{S})),
\end{eqnarray*}
since $|\vec{q}| = \#(\uparrow) + \#(\downarrow)$.   We then see that the unique generator ${\bf x}$ associated to $\mathbb{S}$ satisfies

\[k({\bf x}) = -2{\bf A}_S({\bf x}),\] as desired.
\end{proof}
\bibliography{JacoFest}
\end{document}